\documentclass[12pt,reqno]{amsart}
\usepackage[a4paper]{geometry}
\usepackage{amssymb,amsmath,amsthm,enumerate,verbatim,bbm}
\usepackage[mathscr]{euscript}

\usepackage{xcolor}
\usepackage{mathrsfs}

\DeclareMathOperator{\Ran}{Ran}

\DeclareMathOperator{\Tr}{Tr}

\DeclareMathOperator{\ess}{ess}

\DeclareMathOperator{\const}{const}

\DeclareMathOperator{\Span}{span}

\DeclareMathOperator{\Symb}{Sym}

\newcommand{\abs}[1]{{\lvert#1\rvert}}
\newcommand{\Abs}[1]{\left\lvert#1\right\rvert}
\newcommand{\norm}[1]{\lVert#1\rVert}

\newcommand{\jap}[1]{\langle#1\rangle}

\newcommand{\bbT}{{\mathbb T}}
\newcommand{\bbR}{{\mathbb R}}
\newcommand{\bbC}{{\mathbb C}}
\newcommand{\bbN}{{\mathbb N}}
\newcommand{\bbQ}{{\mathbb Q}}
\newcommand{\bbZ}{{\mathbb Z}}

\newcommand{\bm}{\mathbf m}
\newcommand{\be}{\mathbf e}

\newcommand{\bc}{\mathbf c}

\newcommand{\bL}{\mathbf L}

\newcommand{\bT}{\mathbf T}

\newcommand{\Sch}{\mathbf S}

\numberwithin{equation}{section}
\renewcommand{\[}{\begin{equation}}
\renewcommand{\]}{\end{equation}}

\theoremstyle{plain}
\newtheorem{theorem}{\bf Theorem}[section]
\newtheorem*{theorem*}{Theorem 1.1$'$}
\newtheorem{lemma}[theorem]{\bf Lemma}
\newtheorem{proposition}[theorem]{\bf Proposition}
\newtheorem{corollary}[theorem]{\bf Corollary}

\theoremstyle{definition}

\newtheorem*{definition*}{\bf Definition}

\theoremstyle{remark}
\newtheorem*{remark*}{\bf Remark}

\newtheorem{example}[theorem]{\bf Example}

\DeclareFontFamily{U}{mathx}{\hyphenchar\font45}
\DeclareFontShape{U}{mathx}{m}{n}{<5> <6> <7> <8> <9> <10>
<10.95> <12> <14.4> <17.28> <20.74> <24.88> mathx10}{}
\DeclareSymbolFont{mathx}{U}{mathx}{m}{n}
\DeclareFontSubstitution{U}{mathx}{m}{n}
\DeclareMathAccent{\wc}{0}{mathx}{"71}

\newcommand{\wt}{\widetilde}
\newcommand{\wh}{\widehat}
\newcommand{\eps}{\varepsilon}

\newcommand{\symb}{{\varphi}}
\newcommand{\bsymb}{{\pmb{\varphi}}}
\newcommand{\bmu}{{\pmb{\mu}}}

\begin{document}

\title[Szeg\H{o} Theorem for multiplicative Toeplitz operators]{Szeg\H{o}-type limit theorems for ``multiplicative Toeplitz'' operators and non-F{\o}lner approximations}

\author{Nikolai Nikolski}
\address{Institut de Math\'ematiques de Bordeaux, Universit\'e de Bordeaux, Talence, France, 
and Chebyshev Laboratory, St.Petersburg University}
\email{nikolski@math.u-bordeaux.fr}

\author{Alexander Pushnitski}
\address{Department of Mathematics, King's College London, Strand, London, WC2R~2LS, U.K.}
\email{alexander.pushnitski@kcl.ac.uk}

\subjclass[2010]{47B35}

\keywords{Toeplitz operators, Szeg\H{o} theorem, multiplicative Toeplitz operators, F{\o}lner sequence}

\begin{abstract}
We discuss an analogue of the First Szeg\H{o} Limit Theorem
for multiplicative Toeplitz operators
and highlight the role of the 
multliplicative F{\o}lner condition in this topic. 
\end{abstract}

\date{23 December 2019}

\maketitle

\section{Introduction}
The classical \emph{Toeplitz operators}  
are defined as infinite matrices of the form 
\[
T=\{ c(j-k)\}_{j,k=0}^\infty
\quad 
\text{ on $\ell^2(\bbZ_+)$}, \quad \bbZ_+=\{0,1,2,\dots\},
\label{01}
\]
where  $c$ is a function on the group of integers $\bbZ$. 
It is known that $T$ is bounded on $\ell^2(\bbZ_+)$ if and only if $c$ is the sequence of Fourier coefficients
of a bounded function $\symb\in L^\infty(\bbT)$ (the \emph{symbol} of $T$) on the unit circle
$\bbT$: 
$$
c(k)=\wh \symb(k), \quad k\in\bbZ.
$$ 
This fact is due to O.~Toeplitz, 1911; elementary properties of 
Toeplitz operators can be found, for example, in \cite{Nik2017}.  We will write $T=T(\symb)$. 
Mapping the standard basis in $\ell^2(\bbZ_+)$ onto another orthonormal basis
(in another Hilbert space), one obtains unitarily equivalent realisations of Toeplitz operators.
For example, the Wiener-Hopf integral operators on $L^2(\bbR_+)$ have Toeplitz matrices
with respect to the basis of Laguerre functions. 

The First Szeg\H{o} Limit Theorem (see Theorem~\ref{thm.a1} below) relates the asymptotic 
spectral density of the $N\times N$ truncated Toeplitz matrices
\[
T_N(\symb)=\{\wh \symb(j-k)\}_{j,k=0}^{N-1}
\label{02}
\]
as $N\to\infty$ to the values of the symbol $\symb$.

The subject of this paper is
``\emph{multiplicative Toeplitz operators}''; we use this term to refer to infinite matrices of the form 
$$
\bT=\{\bc(j/k)\}_{j,k=1}^\infty
\quad \text{ on  $\ell^2(\bbN)$,} \quad \bbN=\{1,2,\dots\}.
$$
Here $\bc: \bbQ_+\to\bbC$ is a complex valued function on the set $\bbQ_+$ of positive rationals. 
As we shall see, in full analogy with the classical Toeplitz operators, $\bT$ is bounded 
if and only if $\bc$ is the Fourier transform $\bc=\wh\bsymb$ of a function $\bsymb$ (=symbol)
on the infinite multi-torus $\bbT^\infty$; we will write $\bT=\bT(\bsymb)$. 
We will use boldface font for objects related to the multiplicative case. 

Our aim here is to discuss an analogue of the First Szeg\H{o} Limit Theorem
for multiplicative Toeplitz operators, 
relating the spectral asymptotics of large truncated multiplicative Toeplitz 
matrices to the values of the symbol $\bsymb$.  
In fact, such analogue is a particular case of  \cite[Theorem 11]{Bedos}, which deals with matrices 
constructed from Fourier coefficients of functions on a compact Abelian group $G$.
However, the case $G=\bbT^\infty$ was not explicitly discussed in \cite{Bedos}; 
here we would like to focus on some interesting concrete aspects 
appearing in this case, which are due to the interplay between the multiplicative structure
and the natural order on $\bbN$.

The main new aspect appearing here is the choice of the truncation. 
It turns out that the ``natural" truncation 
$$
\{\wh\bsymb(j/k)\}_{j,k=1}^N
$$
is not admissible, i.e. it does not lead to the ``expected" asymptotic spectral density!
One must instead consider truncations 
$$
\bT_{\sigma_N}(\bsymb)=\{\wh\bsymb(j/k)\}_{j,k\in \sigma_N}
$$
to a sequence of finite subsets of $\sigma_N\subset \bbN$, satisfying the so-called
multiplicative F{\o}lner condition, see \eqref{a5a} below.

The F{\o}lner approximation techniques are well developed in operator theory and in the theory of $C^*$-algebras,
and we refer to \cite{HRS2000} and \cite{ALY2013} for background  and references. 
The question of what kind of asymptotic spectral densities can appear in the non-F{\o}lner 
case is still somewhat murky. One of our aims is to present a variety of examples of 
non-F{\o}lner sequences, both in the additive and multiplicative cases. 
The general conclusion we derive from these examples is that there are no unconditional
spectral limits for finite sections of Toeplitz operators (either additive or multiplicative), apart from trivial cases.

In Section~\ref{sec.a} we recall the classical First Szeg\H{o}  Limit theorem, 
as well as basic facts on F{\o}lner spectral approximations. 
We also provide a few examples of non-F{\o}lner spectral approximations. 
In Section~\ref{sec.aa} we discuss the multiplicative analogues and connections
with other topics. In particular, we state two open problems in Section~\ref{sec.op}. 

For completeness, we give some proofs in the Appendix
although they follow known ideas and deviate little from the construction of \cite{Bedos}. 

\textbf{Notation.}
Throughout the paper, we use $\Sch_1$, $\Sch_2$ to denote the trace class and the 
Hilbert-Schmidt class of compact operators on a given Hilbert space, with the norms
$\norm{\cdot}_{\Sch_1}$ and $\norm{\cdot}_{\Sch_2}$. The operator norm is 
denoted by $\norm{\cdot}$. Notation $\#X$ stands for the number of elements of a finite set $X$.

\section{Asymptotic spectral distributions for Toeplitz operators}\label{sec.a}

\subsection{The First Szeg\H{o} Limit Theorem}
Let $\bbT$ be the unit circle in the complex plane, and let $m$ be the standard 
Lebesgue measure on $\bbT$ with $m(\bbT)=1$. 
For $\varphi\in L^2(\bbT)$, let 
$$
\wh\symb(k)=\int_\bbT\symb(z)\overline{z^k} dm(z), 
\quad k\in\bbZ,
$$
be the Fourier coefficients of $\symb$. 
Let, as above,  
$$
T(\symb)=\{\wh\symb(j-k)\}_{j,k=0}^\infty \quad \text{ in $\ell^2(\bbZ_+)$}
$$
and let  $T_N(\symb)$ be as in  \eqref{02}. 
We have, denoting $e_k(z)=z^k$, 
$$
\wh \symb(j-k)=\jap{\symb\cdot e_j,e_k}_{L^2(\bbT)},
$$
and so the operator  $T(\symb)$ is bounded on $\ell^2(\bbZ_+)$ iff $\symb\in L^\infty(\bbT)$.
Let us also assume that $\symb$ is real-valued; 
then $T(\symb)$ is self-adjoint.

We quote the First Szeg\H{o} Limit Theorem as follows:
\begin{theorem}\cite{Szego}\label{thm.a1}
Let $\symb\in L^\infty(\bbT)$ be a real-valued function,
and let $f$ be a function, continuous on the closed interval $[\ess\inf\symb,\ess\inf\symb]$. 
Then 
\[
\lim_{N\to\infty}\frac1N\Tr f(T_N(\symb))
=
\int_\bbT f(\symb(z))dm(z). 
\label{a2}
\]
\end{theorem}
By standard methods one obtains
\begin{corollary}\label{cr.a2}
If $\lambda_-,\lambda_+\in\bbR$ are such that $\lambda_-<\lambda_+$ and 
\[
m(\symb^{-1}(\{\lambda_-\}))
=
m(\symb^{-1}(\{\lambda_+\}))
=0,
\label{a2b}
\]
then for $\Delta=(\lambda_-,\lambda_+)$  we have
\[
\lim_{N\to\infty}
\frac1N\#\{k: \lambda_k(T_N(\symb))\in\Delta\}
=
m(\{z\in\bbT: \symb(z)\in\Delta\}),
\label{a2a}
\]
where $\{\lambda_k(T_N(\symb))\}_{k=1}^N$ are the eigenvalues of $T_N(\symb)$. 
\end{corollary}

\begin{remark*}
\begin{enumerate}[1.]
\item
The assumptions on  $\symb$ and  $f$ in Theorem~\ref{thm.a1}
can be considerably relaxed. 
Our purpose here is only to set the scene and to give the context for the multiplicative analogues below, and so 
we are quoting the simplest version of the First Szeg\H{o} Limit  Theorem. 
For a more in-depth discussion of this topic, see e.g. \cite[Section 5.4]{BS} or \cite[Section 2.7.7]{Simon}.

\item
Formula \eqref{a2a} can be rephrased by saying that the asymptotic spectral density of $\{T_N(\symb)\}_{N=1}^\infty$
is given by the push-forward $m_\symb$ of the Lebesgue measure $m$  by the map $\symb$, 
i.e. $m_\symb(\Delta)=m(\symb^{-1}(\Delta))$.
The measure $m_\symb$ contains some information about the spectral measure $E_{M(\symb)}$ of the operator
$M(\symb)$ of multiplication by $\symb$; for example, $m_\symb$ and $E_{M(\symb)}$
are mutually absolutely continuous. 
However, some other spectral invariants of $M(\symb)$, such as the
spectral multiplicity function, are not determined by $m_\symb$. 
Moreover, $m_\symb$ has little in common with the spectral measure of the Toeplitz operator $T(\symb)$
(which is always purely absolutely continuous, unless $\symb$ is constant). 
Recall that the spectra of $T(\symb)$ and $M(\symb)$ are, in general, distinct, and, although the finite sections 
$T_N(\symb)=(T(\symb)_{ij})_{0\leq i,j\leq N}=(M(\symb)_{ij})_{0\leq i,j\leq N}$ coincide, 
 the 
spectrum of $T_N(\symb)$, as a set, converges to the spectrum of $T(\symb)$. 
See \cite[Section 5.6.3]{Nik2017} for a discussion of this phenomenon.

\item
Formally, \eqref{a2a} is a particular case of \eqref{a2} with $f=\chi_{\Delta}$. 
The (easy) justification of the limiting process from continuous $f$ to 
discontinuous $\chi_{\Delta}$ requires the additional assumption \eqref{a2b}. 

\item
An important particular case of Theorem~\ref{thm.a1} is $f(x)=\log x$ (provided the 
range of $\symb$ is a compact set in $(0,\infty)$). 
Then the theorem can be written as
\[
\lim_{N\to\infty}\bigl(\det T_N(\symb))^{1/N}
=
\exp\biggl(\int_{\bbT} \log \symb(z)dm(z) \biggr).
\label{a3}
\]
In fact, this formula remains true for Toeplitz matrices 
$$
T_N(\mu)=\{\wh\mu(j-k)\}_{j,k=0}^{N-1},
$$ 
associated with any Borel measure $\mu\geq0$ with the Radon-Nikodym 
decomposition $\mu=\symb dm+\mu_{\text{sing}}$.
The corresponding Toeplitz operator $T(\mu)$ does not need to be bounded
on $\ell^2(\bbZ_+)$, nor $\log \symb$ has to be integrable. 
For discussions and further results, see \cite[p.141]{Simon} and \cite{Nik2017}. 

\item
We do not discuss the strong Szeg\H{o} theorem, which deals with the second term
in \eqref{a2}.
\end{enumerate}
\end{remark*}

\subsection{Extension to F{\o}lner sequences}\label{sec.a3}
In order to prepare the ground for the multiplicative set-up below, 
here we briefly discuss an extension of the classical Szeg\H{o} theorem.

Let $\sigma\subset\bbN$ be a finite set. 
We denote by $T_\sigma(\symb)$ the finite section of the infinite matrix $T(\symb)$
corresponding to the indices restricted to the set $\sigma$, i.e. 
$$
T_\sigma(\symb)=\{\wh\symb(j-k)\}_{j,k\in\sigma}. 
$$
We will say that a sequence $\{\sigma_N\}_{N=1}^\infty$ of subsets of $\bbZ_+$
is an \emph{additive F{\o}lner sequence,} if for any $n\in\bbZ_+$ we have
\[
\frac{\#\{k\in\sigma_N: k+n\in \sigma_N\}}{\#\sigma_N}\to1, \quad \text{ as } N\to\infty.
\label{fo1}
\]
F{\o}lner sequences can be defined in the context of any semigroup acting on a countable set;
here we have the action of the additive semigroup $\bbZ_+$ on itself. 
Clearly, $\sigma_N=\{0,1,\dots,N-1\}$ is an additive F{\o}lner sequence.

\begin{theorem}\label{thm.a1a}
Let $\symb\in L^\infty(\bbT)$ be a real-valued function, and let 
$\{\sigma_N\}_{N=1}^\infty$ 
be an additive F{\o}lner sequence of finite subsets of $\bbZ_+$. 
Then for any function $f$, continuous on $[\ess\inf\symb,\ess\sup\symb]$, we have
\[
\lim_{N\to\infty}
\frac1{\#\sigma_N}\Tr f(T_{\sigma_N}(\symb))
=
\int_{\bbT}f(\symb(z))dm(z).
\label{a3a}
\]
\end{theorem}
This theorem was proved in the doctoral dissertation of D.~B.~SeLegue, 
following general ideas of W.~Arveson (see \cite{HRS2000,ALY2013} for references). 
It can also be regarded as a particular case of the more general Theorem~\ref{thm.bedos}, see below. 
Later on, similar and related constructions were considered by many authors \cite{Otte,BottcherOtte}.

\begin{remark*}
In fact, SeLegue's construction gives asymptotic spectral density for a much wider set of 
operators, namely for all $T$ from the $C^*$-algebra $A$ generated by Toeplitz operators
(see \cite{HRS2000} for an excellent presentation). 
Namely, if $P_\sigma$ is the orthogonal projection onto $\Span\{z^k: k\in\sigma\}$, then 
for every self-adjoint $T\in A$ and every $f\in C(\bbR)$ 
$$
\lim_{N\to\infty}
\frac1{\#\sigma_N}\Tr f(P_{\sigma_N}T |_{\Ran P_{\sigma_N}})
=
\int_{\bbT}f\circ\Symb(T)\, dm,
$$
where $\Symb(T)\in L^\infty(\bbT)$ stands for the symbol of $T$, which is defined through the 
continuous extension of the elementary symbol map 
$$
\Symb: \sum_i\prod_j T(\varphi_{i,j})\mapsto \sum_i\prod_j \varphi_{i,j}, \quad \varphi_{i,j}\in L^\infty(\bbT),
$$
see \cite{Nik2017} for the details related to the symbol map.
\end{remark*}

\subsection{Sharpness of the F{\o}lner condition}

It is easy to see that Theorem~\ref{thm.a1a} is sharp in the following sense. 
\begin{proposition}\label{prp.a1b}
Let $\{\sigma_N\}_{N=1}^\infty$ be a sequence of finite subsets of $\bbZ_+$ such 
that for some $n\in\bbZ_+$, the additive F{\o}lner condition \eqref{fo1} fails. 
Then there exists a real-valued symbol $\symb\in C(\bbT)$ 
such that the conclusion \eqref{a3a} of Theorem~\ref{thm.a1a} fails already for $f(x)=x^2$. 
In fact, one can take $\symb(z)=z^n+\overline{z}^n$. 
\end{proposition}
\begin{proof} 
Let $\symb(z)=z^n+\overline{z}^n$ and $f(x)=x^2$;  then the r.h.s. of \eqref{a3a} is 
$$
\int_{\bbT}\symb(z)^2 dm(z)=2. 
$$
Now let us consider the left hand side. 
We have
\begin{align*}
\Tr T_{\sigma_N}(\symb)^2
&=
\sum_{j\in\sigma_N}\sum_{k\in\sigma_N}\wh\symb(j-k)\wh\symb(k-j)
=
\sum_{j\in\sigma_N}\sum_{k\in\sigma_N}\abs{\wh\symb(j-k)}^2
\\
&=
\sum_{r\in\bbZ}\abs{\wh\symb(r)}^2\#\{k\in \sigma_N: k+r\in\sigma_N\}
\\
&=
\#\{k\in\sigma_N: k+n\in\sigma_N\}
+
\#\{k\in\sigma_N: k-n\in\sigma_N\}.
\end{align*}
Observe that by the change of parameter $k'=k-n$, the two terms in the r.h.s. here are equal
to one another. We conclude that 
$$
\frac{\Tr T_{\sigma_N}(\symb)^2}{\#\sigma_N}
=
2
\frac{\#\{k\in \sigma_N: k+n\in \sigma_N\}}{\#\sigma_N},
$$
and by assumption the r.h.s. does not converge to $2$ as $N\to\infty$. 
Thus, \eqref{a5} fails. 
\end{proof}

\subsection{Non-F{\o}lner sequences}

Here we discuss some examples of  sequences
that do NOT satisfy the additive F{\o}lner condition. 
Our purpose is to illustrate two possibilities: convergence to a ``wrong" limit and divergence. 

\begin{example}\label{ex1}
Let 
$$
\sigma_N=\{0,2,4,\dots,2N\}.
$$
Then $\{\sigma_N\}_{N=1}^\infty$ is not an additive F{\o}lner sequence. 
It is easy to see that in this case the limit in \eqref{a3a} exists, 
but is given by the modified expression: 
$$
\lim_{N\to\infty}\frac1{\#\sigma_N}\Tr f(T_{\sigma_N}(\symb))
=
\int_{\bbT}f(\symb_{2}(z))dm(z), 
$$
where $\symb_{2}$ is the even part of $\symb$, 
$$
\symb_{2}(z)=\frac12(\symb(z)+\symb(-z)).
$$
More generally, if $\ell\geq2$, then setting $\sigma_N=\{0,\ell,2\ell,3\ell,\dots,N\ell \}$ yields
$$
\lim_{N\to\infty}\frac1{\#\sigma_N}\Tr f(T_{\sigma_N}(\symb))
=
\int_{\bbT}f(\symb_{\ell}(z))dm(z), 
$$
where 
$$
\symb_\ell(z)=\frac1\ell\sum_{j=0}^{\ell-1} \symb(ze^{2\pi i j/\ell}).
$$
\end{example}
One can object that in this example the union $\cup_{N=1}^\infty \sigma_N$ is not the whole of $\bbZ_+$. 
However, it is easy to modify this example  to fix this problem. We use the following 

\begin{lemma}[Finite subsets are negligible]\label{lma.a6}
Let $\symb\in L^\infty(\bbT)$, let $\{\sigma_N\}_{N=1}^\infty$ be a (possibly non-F{\o}lner) 
sequence of finite subsets of $\bbZ_+$ with $\#\sigma_N\to\infty$, and let $\sigma\subset\bbZ_+$ be a finite set;
denote $\sigma_N'=\sigma_N\cup\sigma$. Then for every polynomial $f$, the limits 
$$
\lim_{N\to\infty} \frac1{\#\sigma_N}\Tr f(T_{\sigma_N}(\symb))
\quad \text{ and }\quad 
\lim_{N\to\infty} \frac1{\#\sigma_N'}\Tr f(T_{\sigma_N'}(\symb))
$$
exist or do not exist simultaneously; if they exist, their values coincide. 
\end{lemma}
\begin{proof}
If $f(x)=\const$, the statement is obvious. 
Let us consider the case $f(x)=x^m$, $m\geq1$.
For brevity, let us denote  $A_N=T_{\sigma_N'}(\symb)=\{\wh\symb(j-k)\}_{j,k\in\sigma_N'}$
(this is a matrix of the size  $(\#\sigma_N')\times(\#\sigma_N')$), and let $B_N=\{(B_N)_{j,k}\}_{j,k\in\sigma_N'}$ be the matrix of the same size
defined as follows:
$$
(B_N)_{j,k}=
\begin{cases}
\wh \symb(j-k), & j,k\in\sigma_N,
\\ 0, &\text{at least one of $j,k$ is not in $\sigma_N$.}
\end{cases}
$$
By this definition, we have 
$$
\Tr (T_{\sigma_N}(\symb))^m=\Tr B_N^m. 
$$
For the matrix $C_N=A_N-B_N$ we obtain
$$
\norm{C_N}_{\Sch_2}^2
=\sum_{j,k}\abs{\wh\symb(j-k)}^2, 
$$
where the sum is taken over $j,k\in\sigma_N'$ such that at least one of the indices $j,k$ is not in  $\sigma_N$. 
It follows that 
$$
\norm{C_N}_{\Sch_2}^2
\leq (\#\sigma)\sum_{j\in\bbZ}\abs{\wh\symb(j)}^2=C<\infty.
$$
Furthermore, we have 
$$
A_N^m-B_N^m=A_N^{m-1}(A_N-B_N)+A_N^{m-2}(A_N-B_N)B_N+\dots+(A_N-B_N)B_N^{m-1}, 
$$
and therefore 
\begin{align}
\norm{A_N^m-B_N^m}_{\Sch_1}
\leq&
\norm{A_N^{m-1}}\norm{A_N-B_N}_{\Sch_1}
+
\norm{A_N^{m-2}}\norm{B_N}\norm{A_N-B_N}_{\Sch_1}
\notag
\\
&+\dots+
\norm{B_N^{m-1}}\norm{A_N-B_N}_{\Sch_1}
\notag
\\
\leq&
m\max(\norm{A_N}^{m-1},\norm{B_N}^{m-1})
\norm{A_N-B_N}_{\Sch_1}. 
\label{a7}
\end{align}
Finally, 
$$
\norm{A_N-B_N}_{\Sch_1}
=
\norm{C_N}_{\Sch_1}
\leq
(\#\sigma'_N)^{1/2}\norm{C_N}_{\Sch_2},
$$
and so, putting this together, 
$$
\frac1{\#\sigma_N}
\Tr (A_N^m-B_N^m)=o(1),\quad N\to\infty.
$$
This proves the required statement for $f(x)=x^m$. 
The general case follows by  taking linear combinations.
\end{proof}

\begin{example}[Example \ref{ex1} modified]
For simplicity of notation, consider $\ell=2$. 
Let $\sigma_N$ be as in Example~\ref{ex1}, i.e. 
$$
\sigma_N=\{0,2,4,\dots,2N\}.
$$
For every $k\in\bbN$, let 
$$
\rho_k=\{0,\dots,k-1\}\cup\sigma_{N(k)}, 
$$
where the sequence $N(k)$ is chosen as follows. 
For every fixed $k$, by the above lemma we can choose $N(k)$ sufficiently large so that
$$
\Abs{\frac1{\#\sigma_{N(k)}}\Tr (T_{\sigma_{N(k)}}(\symb))^m
-
\frac1{\#\rho_{k}}\Tr (T_{\rho_{k}}(\symb))^m}
\leq 1/k
$$
for all $m=1,\dots,k$. 
Then for all polynomials $f$, we have 
$$
\lim_{k\to\infty}
\frac1{\#\rho_{k}}\Tr f(T_{\rho_{k}}(\symb))
=
\lim_{k\to\infty}
\frac1{\#\sigma_{N(k)}}\Tr f(T_{\sigma_{N(k)}}(\symb))
=
\int_{\bbT}f(\symb_{2}(z))dm(z), 
$$
and $\cup_{k=1}^\infty \rho_k=\bbZ_+$. 
\end{example}

\begin{example}
Let $\sigma_N=\{0,1,\dots,N-1\}$, and let $\rho_N$ be as in the previous 
example. Consider the sequence of sets 
$$
\rho_{N(1)}, \sigma_{N(2)}, \dots, \rho_{N(2j-1)}, \sigma_{N(2j)},\dots,
$$
where  $N(k)\nearrow\infty$ sufficiently fast. 
This is not an additive F{\o}lner sequence, although it exhausts $\bbZ_+$ and can be made monotone. 
The corresponding expression in the l.h.s. of \eqref{a3a} does NOT converge
to any limit, except in trivial cases. 
\end{example}

\begin{example}
Let us consider what happens if the sequence $\sigma_N$ is sufficiently sparse.    
Let $\sigma_N=\{1,3,3^2,\dots,3^{N-1}\}$. Then 
$$
T_{\sigma_N}=\{\wh\symb(3^j-3^k)\}_{j,k=0}^{N-1}=\wh\symb(0)I_N+A_N,
$$
where $I_N$ is the $N\times N$ identity matrix and $A_N$ is an $N\times N$ section of an infinite
Hilbert-Schmidt matrix. From here it is easy to conclude that 
$$
\lim_{N\to\infty}
\frac1{\#\sigma_N}\Tr f(T_{\sigma_N}(\symb))
=
f(\wh \symb(0))
$$
for all continuous functions $f$. 
Again, using Lemma~\ref{lma.a6}, it is easy to modify this example so that $\sigma_N$ satisfy
$\sigma_N \nearrow\bbZ_+$. 
\end{example}

\subsection{Extension to compact Abelian groups}

Here we recall an extension of the First Szeg\H{o} Limit Theorem
due to E.~B\'edos \cite{Bedos}, see also \cite{HRS2000,ALY2013}
for more details and references. 
Let $G$ be a compact abelian group (with additive notation), equipped
with the normalised Haar measure $m$, and let $\Gamma=\wh G$ be its
(discrete) character group. Given a function (symbol) $\varphi\in L^\infty(G)$, 
let  $M(\varphi)$ be the operator of multiplication by $\varphi$ in $L^2(G)$. 
Next, for a finite subset $\sigma\subset \Gamma$, we define the orthogonal projection
onto the subspace of polynomials with frequencies in $\sigma$: 
$$
P_\sigma\bigl(\sum_{\gamma\in \Gamma} a_\gamma \gamma\bigr)
=
\sum_{\gamma\in\sigma}a_\gamma \gamma,
$$
where $a_\gamma\in\bbC$ is any family of complex numbers with a finite support. 
Now let $T_\sigma(\varphi)$ be the finite truncation of $M(\varphi)$: 
$$
T_\sigma(\varphi)=P_\sigma M(\varphi)|_{P_\sigma L^2(G)}. 
$$
The matrix of $T_\sigma(\varphi)$ in the orthogonal basis of monomials in $\sigma$ is
$\{\wh\varphi(\alpha-\beta)\}_{\alpha,\beta\in\sigma}$. 
The \emph{F{\o}lner condition} for a sequence of finite subsets $\sigma_N\subset\Gamma$
is defined similarly to the case of $\Gamma=\bbZ$ above:
\[
\frac{\#\{\gamma\in \sigma_N: \gamma+\alpha\in\sigma_N\}}{\#\sigma_N}\to1, \quad \text{ as $N\to\infty$,}
\label{a4}
\]
for every $\alpha\in\Gamma$.

\begin{theorem}\label{thm.bedos}\cite[Theorem 11]{Bedos}
Let $\symb\in L^\infty(G)$ be a real-valued function, let 
$\{\sigma_N\}_{N=1}^\infty$ 
be a  F{\o}lner sequence of finite subsets of $\Gamma=\wh G$ and let $T_\sigma(\symb)$ be as defined above.
Then for any function $f$, continuous on $[\ess\inf\symb,\ess\sup\symb]$, we have
$$
\lim_{N\to\infty}
\frac1{\#\sigma_N}\Tr f(T_{\sigma_N}(\symb))
=
\int_{G}f\circ \symb\, dm.
$$
\end{theorem}

\subsection{Multi-dimensional Toeplitz operators}

Let us briefly discuss an important particular case of the previous theorem: 
$G=\bbT^d$, $d>1$. Here $\Gamma=\bbZ^d$, and for 
$\sigma\subset \bbZ_+^d$, the matrix of $T_\sigma(\varphi)$
is a finite section of the multi-dimensional Toeplitz operator
$$
T(\symb)=\{\wh \symb(j-k)\}_{j,k\in\bbZ_+^d} \quad \text{ on $\ell^2(\bbZ_+^d)$.}
$$
In this context, we would like to mention the following multi-dimensional 
extension of Example~\ref{ex1}. 

\begin{example}\label{exa11}
Let $\Gamma'\subset \bbZ^d$ be a subgroup, and let $\{\sigma_N\}_{N=1}^\infty$
be a F{\o}lner sequence in $\Gamma'$ (i.e. $\sigma_N\subset \Gamma'$ and 
\eqref{a4} holds for every $\alpha\in \Gamma'$). 
Then it is easy to check the following modified spectral density formula:
\[
\lim_{N\to\infty}\frac1{\#\sigma_N}\Tr f(T_{\sigma_N}(\varphi))=\int_{\bbT^d}f\circ\varphi_{\Gamma'}dm_d.
\label{a6}
\]
Here $\varphi_{\Gamma'}$ is a ``projection of $\varphi$ onto $\Gamma'$", 
$$
\varphi_{\Gamma'}(z)=\int_{(\Gamma')^\perp}\varphi(z\zeta)dm_{(\Gamma')^\perp}(\zeta),
$$
where
$(\Gamma')^\perp=\{\zeta\in \bbT^d: \gamma(\zeta)=1 \quad \forall \gamma\in\Gamma'\}$ 
stands for the group of trivial characters on $\Gamma'$.  
Indeed, for every $\varphi\in L^\infty(\bbT^d)$, the function 
$\varphi_{\Gamma'}$ is invariant under the action of $(\Gamma')^\perp$, 
i.e. $\varphi_{\Gamma'}(z\zeta)=\varphi_{\Gamma'}(z)$ for every $z\in\bbZ^d$ 
and $\zeta\in(\Gamma')^{\perp}$, so that $\varphi_{\Gamma'}$ is, in fact, defined on the quotient group 
$\bbT^d/(\Gamma')^\perp=\widehat{(\Gamma')}$ (and the same is true for $f\circ \varphi_{\Gamma'}$). 
Moreover, 
$$
\int_{\widehat{(\Gamma')}} f\circ\varphi_{\Gamma'}dm_{\widehat{(\Gamma')}}
=
\int_{\bbT^d} f\circ \varphi_{\Gamma'} dm_d
$$
(see, for example, \cite[Section 30]{We1940}). 
Applying Theorem~\ref{thm.bedos} to $\varphi_{\Gamma'}$ and $\{\sigma_N\}_{N=1}^\infty$, we get \eqref{a6}.

Observe also that if $\symb$ does not depend on the variables $\zeta\in (\Gamma')^\perp$, i.e. 
$$
\symb(z\zeta)=\symb(\zeta), \quad \forall \zeta\in (\Gamma')^\perp,
$$
then $\varphi_{\Gamma'}=\varphi$, and so the right hand side in \eqref{a6} coincides
with the standard expression.

Using Lemma~\ref{lma.a6}, it is easy to modify this example to obtain 
a (non-F{\o}lner) sequence $\sigma_N$ which satisfies \eqref{a6} and also $\sigma_N\nearrow\bbZ_+^d$. 
\end{example}

\section{Multiplicative Toeplitz operators}\label{sec.aa}

\subsection{The infinite multi-torus $\bbT^\infty$}
Let $\bbT^\infty$ be the Cartesian product of countably many copies of $\bbT$,
equipped with the product topology. 
We will denote elements of $\bbT^\infty$ by $z=(z_1,z_2,\dots)$, where
$z_j\in\bbT$ for each $j$. 
Clearly, $\bbT^\infty$ is a compact Abelian group, 
and therefore there exists a unique Haar measure $\bm$ on $\bbT^\infty$. 
We will need the spaces $L^\infty(\bbT^\infty)$ and $L^2(\bbT^\infty)$, considered
with respect to this measure. 

Let $\bbZ^{(\infty)}$ be the set of all multi-indices $\alpha=(\alpha_1,\alpha_2,\dots)$, 
such that $\alpha_j\in\bbZ$ for each $j$ and $\alpha_j=0$ for all but finitely many 
indices $j$. The subset $\bbZ_+^{(\infty)}\subset \bbZ^{(\infty)}$ 
is defined by the additional condition $\alpha_j\geq0$ for all $j$.  
For $z\in \bbT^\infty$ and $\alpha\in\bbZ^{(\infty)}$ we will use the 
``multipower" notation $z^\alpha$ for the product
$$
z^\alpha=z_1^{\alpha_1}z_2^{\alpha_2}\cdots\, .
$$
We also denote by $\be_\alpha$ the corresponding function on $\bbT^\infty$, $\be_\alpha(z)=z^\alpha$. 
With this notation, $\{\be_\alpha\}_{\alpha\in\bbZ^{(\infty)}}$ is a complete orthonormal
set in $L^2(\bbT^\infty)$. 

Let $p_1,p_2,\dots$ be the ordered sequence of all primes. 
Every natural number $n$ can be written uniquely as a product 
\[
n=p^\alpha:=p_1^{\alpha_1}p_2^{\alpha_2}\cdots,\quad \alpha=\alpha(n)\in\bbZ_+^{(\infty)}\, ,
\label{mp}
\]
and similarly every
positive rational $q$ can be written uniquely as $q=p^\alpha$ for $\alpha=\alpha(q)\in \bbZ^{(\infty)}$. 

For $\bsymb\in L^2(\bbT^\infty)$, its Fourier coefficients are the coordinates with respect to the orthonormal basis 
$\{\be_\alpha\}_{\alpha\in\bbZ^{(\infty)}}$, i.e. 
$$
\jap{\bsymb, \be_\alpha}_{L^2(\bbT^\infty)}=\int_{\bbT^\infty}\bsymb(z)\overline{z^{\alpha}}d\bm(z)\, ,
\quad \alpha\in \bbZ^{(\infty)}\, . 
$$
One can label these Fourier coefficients either by 
multi-indices $\alpha\in \bbZ^{(\infty)}$, or by positive rationals $q=p^\alpha$. 
We choose the latter option and denote
$$
\wh\bsymb(p^\alpha)=\jap{\bsymb, \be_\alpha}_{L^2(\bbT^\infty)},
\quad 
\alpha\in\bbZ^{(\infty)}\, . 
$$

\subsection{Bohr's lift}

\emph{Bohr's lift} $B$ is a linear one-to-one correspondence between 
(appropriate spaces of) almost-periodic functions on $\bbR$ and functions on $\bbT^\infty$. 
For almost-periodic finite linear combinations
$$
f(t)=\sum_{\alpha}c_\alpha p^{it\alpha}, \quad p^{it\alpha}:=(p^\alpha)^{it},
$$
the Bohr's lift $Bf$ is defined as the function on $\bbT^\infty$, given by 
$$
Bf(z)=\sum_{\alpha} c_\alpha z^\alpha, \quad z\in \bbT^\infty.
$$
In other words, the inverse $B^{-1}$ is defined on polynomials $\bsymb$ by 
$$
B^{-1}\bsymb(t)=\bsymb(p(it)), \quad \text{ where }p(it)=(p_1^{it}, p_2^{it}, \dots).
$$
It is important that $B$ is not only linear, but also multiplicative
(i.e. it maps products to products). 
By Bohr's lemma (see e.g. \cite[Theorem 6.5.1]{QQ}), 
we have
\[
\lim_{T\to\infty}\frac1{2T}\int_{-T}^T f(t)dt
=
\int_{\bbT^\infty} Bf(z)d\bm(z). 
\label{kr}
\]
Applying this to $\abs{f}^2$ in place of $f$, we see that Bohr's lift extends to a unitary map 
$$
B: AP^2(S)\to L^2(\bbT^\infty);
$$
here $AP^2(S)$ is the Besicovitch $L^2$-space of almost periodic functions
with Fourier spectrum in the set 
$$
S=\bigl\{\sum_j \alpha_j \log p_j: \alpha\in \bbZ^{(\infty)}\bigr\}.
$$

\subsection{Multipicative Toeplitz operators}
Let $\bsymb\in L^\infty(\bbT^\infty)$ be a complex-valued function. 
The \emph{multliplicative Toeplitz operator with symbol $\bsymb$}
is defined by 
$$
\bT(\bsymb)=\{\wh\bsymb(j/k)\}_{j,k\in\bbN}
\quad
\text{ in $\ell^2(\bbN)$.}
$$
Multiplicative Toeplitz operators are of some interest in multiplicative number theory, 
see e.g. \cite{Hilberdink} and references therein; 
this is mainly due to their connection with the Dirichlet convolution. 
Historically, multiplicative Toeplitz operators have appeared as early as in the 1938 paper \cite{Toeplitz}
by Toeplitz (so the name ``multiplicative Toeplitz operators" is not entirely inappropriate!). 
But the paper \cite{Toeplitz} seems to be little known even to experts in the area. 
In more recent time, the study of multiplicative Toeplitz matrices has been revived and put firmly into the context
of functional analysis in \cite{HLS}.

Similarly to the additive case, we have the relation
$$
\wh\bsymb(p^\alpha/p^\beta)=\jap{\bsymb\cdot \be_\alpha,\be_\beta}_{L^2(\bbT^\infty)}, 
\quad
\alpha,\beta\in\bbZ_+^{(\infty)}.
$$
It follows that $\bT(\bsymb)$ is bounded in $\ell^2(\bbN)$ if $\bsymb$ is bounded and, as in the classical ``additive'' case, 
$$
\norm{\bT(\bsymb)}=\norm{\bsymb}_{L^\infty(\bbT^\infty)}.
$$
Conversely, if an infinite matrix $\{\mathbf c(j/k)\}_{j,k\in\bbN}$ is bounded on $\ell^2(\bbN)$, then 
necessarily $\mathbf c=\wh\bsymb$ for a (unique) symbol $\bsymb\in L^\infty(\bbT^\infty)$. 
These facts were proven in \cite{Toeplitz} in the language of almost-periodic functions.

Furthermore,  if $\bsymb$ is real-valued, then 
$$
\wh\bsymb(1/q)=\overline{\wh\bsymb(q)}, \quad q\in\bbQ_+,
$$
and so $\bT(\bsymb)$ is self-adjoint.

\subsection{The First Szeg\H{o} Limit Theorem}

Let $\sigma\subset\bbN$ be a finite set. 
We denote by $\bT_\sigma(\bsymb)$ the truncation of 
$\bT(\bsymb)$ corresponding to the indices being restricted onto the subset $\sigma$: 
$$
\bT_\sigma(\bsymb)=\{\wh\bsymb(j/k)\}_{j,k\in\sigma}. 
$$
We will say that a sequence $\{\sigma_N\}_{N=1}^\infty$ of subsets of $\bbN$ is a 
\emph{multiplicative F{\o}lner sequence}, if for any $n\in\bbN$, 
\[
\frac{\#\{k\in\sigma_N: nk\in \sigma_N\}}{\#\sigma_N}
\to 1
\quad 
\text{ as $N\to\infty$.}
\label{a5a}
\]
Here we have the multiplicative semigroup $\bbN$ acting on itself. 

As mentioned in the introduction, the following result is a particular case of Theorem~\ref{thm.bedos},
although this particular case is not explicitly discussed in \cite{Bedos}. 

\begin{theorem}\label{thm.a2}
Let $\bsymb\in L^\infty(\bbT^\infty)$ be a real-valued function, and let 
$\{\sigma_N\}_{N=1}^\infty$ 
be a multiplicative F{\o}lner sequence. 
Then for any function $f$, continuous on 
$[\ess\inf\bsymb,\ess\sup\bsymb]$,
we have
\[
\lim_{N\to\infty}
\frac1{\#\sigma_N}\Tr f(\bT_{\sigma_N}(\bsymb))
=
\int_{\bbT^\infty}f(\bsymb(z))d\bm(z).
\label{a5}
\]
\end{theorem}

\begin{corollary}\label{cr.4}
If $\lambda_-,\lambda_+\in\bbR$ are such that $\lambda_-<\lambda_+$ and 
\[
\bm(\bsymb^{-1}(\{\lambda_-\}))
=
\bm(\bsymb^{-1}(\{\lambda_+\}))
=0,
\label{a5b}
\]
then for $\Delta=(\lambda_-,\lambda_+)$ we have
\[
\lim_{N\to\infty}
\frac{\#\{k: \lambda_k(\bT_{\sigma_N}(\bsymb))\in\Delta\}}{\#\sigma_N}
=
\bm(\{z\in\bbT^\infty: \bsymb(z)\in\Delta\})
\label{a5c}
\]
where $\{\lambda_k(\bT_{\sigma_N}(\bsymb))\}_{k=1}^{\#{\sigma_N}}$ 
are the eigenvalues of $\bT_{\sigma_N}(\bsymb)$.
\end{corollary}

\subsection{Multiplicative F{\o}lner and non-F{\o}lner sequences}
\mbox{} \newline

\textbf{(a)}
The most ``natural'' sequence of sets
\[
\sigma_N=\{1,2,\dots,N\}
\label{a6a}
\]
is NOT multiplicative F{\o}lner!
Indeed, it is easy to see that for all $n\in\bbN$
$$
\frac{\#\{k\in\sigma_N: nk\in \sigma_N\}}{\#\sigma_N}\to \frac1n, \quad N\to\infty.
$$
In \cite{Balazard}, the author proves the \emph{existence} of the limit 
in the left hand side of \eqref{a5} for $\sigma_N=\{1,2,\dots,N\}$ 
when $f(x)=\log x$ and $\wh\bsymb$ has some rather special arithmetic properties. 
This limit is not evaluated in \cite{Balazard}.

\textbf{(b)}
Let 
$$
\sigma_N=\{p^\alpha: 0\leq \alpha_j\leq A_j^{(N)}\}, 
$$
where for each $N$, only finitely many coefficients $A_j^{(N)}$ are non-zero, and 
for each $j$, 
$$
A_j^{(N)}\to \infty \text{ as }N\to\infty.
$$
Then it is easy to see that $\sigma_N$ is multiplicative F{\o}lner. 

\textbf{(c)}
More generally, the question of whether $\{\sigma_N\}_{N=1}^\infty$ is multiplicative F{\o}lner 
depends the geometry of the subsets
$$
\Sigma_N=\{\alpha\in \bbZ_+^{(\infty)}: p^\alpha\in\sigma_N\}
$$
of the infinite-dimensional lattice $\bbZ_+^{(\infty)}$
as $N\to\infty$. 
The subsets $\Sigma_N$ for $\sigma_N=\{1,2,\dots,N\}$ behave very irregularly. 

\textbf{(d)}
Let us modify Example~\ref{exa11} to fit this framework. 
Let $d\geq1$; 
for simplicity of notation, we consider the simplest embedding $\iota: \bbZ_+^d\to\bbZ_+^{(\infty)}$, 
realised by 
$$
\iota(j_1,\dots,j_d)=(j_1,\dots,j_d,0,0,\dots). 
$$
We identify $\bbZ_+^d$ with $\iota (\bbZ_+^d)$. 
Let $\sigma_N\subset \bbN$ be such that the corresponding subsets
$\Sigma_N$ satisfy $\Sigma_N\subset \bbZ_+^d$, and 
$\{\Sigma_N\}_{N=1}^\infty$ is an (additive) F{\o}lner sequence for $\bbZ_+^d$.
Then in place of \eqref{a5} we have
$$
\lim_{N\to\infty}
\frac1{\#\sigma_N}\Tr f(\bT_{\sigma_N}(\bsymb))
=
\int_{\bbT^d}f(\bsymb_d(z))dm(z),
$$
where 
$$
\bsymb_d(z)=\int_{\bbT^\infty}\bsymb(z_1,\dots,z_d,\zeta_1,\zeta_2,\dots) d\bm(\zeta).
$$
Using Lemma~\ref{lma.a6}, we can again modify this example so that $ \sigma_N\nearrow\bbN$.

\subsection{Sharpness of the F{\o}lner condition}
Exactly as in the additive case, Theorem~\ref{thm.a2} is sharp in the following sense. 
\begin{proposition}
Let $\{\sigma_N\}_{N=1}^\infty$ be a sequence of finite subsets of $\bbN$ such 
that for some $n\in\bbN$, the multiplicative F{\o}lner condition \eqref{a5a} fails. 
Then there exists a real-valued symbol $\bsymb\in C(\bbT^\infty)$ 
such that the conclusion \eqref{a5} of Theorem~\ref{thm.a2} fails for $f(x)=x^2$. 
In fact, one can take $\bsymb(z)=z^\alpha+z^{-\alpha}$, where $n=p^\alpha$. 
\end{proposition}
\begin{proof} 
The proof is a verbatim translation of the proof of Proposition~\ref{prp.a1b} written for the group $\bbQ_+$. 
\end{proof}

\subsection{Comments, corollaries and possible extensions}\label{sec.b7}
\mbox{}\newline

\textbf{(a)}
Using Bohr's lemma \eqref{kr}, one can rewrite the 
conclusion of Theorem~\ref{thm.a2} 
as
$$
\lim_{N\to\infty}\frac1{\#\sigma_N}\Tr f(\bT_{\sigma_N}(\bsymb))
=
\lim_{T\to\infty}
\frac1{2T}
\int_{-T}^T f(\bsymb(p(it)))dt, 
$$
where $p(it)=(p_1^{it}, p_2^{it}, \dots)$.

\textbf{(b)}
Assumption $\bsymb\in L^\infty(\bbT^\infty)$ of 
Theorem~\ref{thm.a2} can be considerably relaxed in the spirit of \cite[Section 2.7.7]{Simon}. 
Furthermore, this theorem also holds true for any complex-valued $\bsymb$, 
if $f$ is a uniform limit of polynomials on the closed convex hull of the range of $\bsymb$. 

\textbf{(c)}
In addition to the previous item, 
we give another non-selfadjoint analogue of Theorem~\ref{thm.a2}. 
Here the symbol $\bsymb$ is not assumed to be real-valued, but instead of the eigenvalues
of $\bT_{\sigma_N}(\bsymb)$ one considers the singular values.

\begin{theorem}\label{thm.a3}
Let $\bsymb\in L^\infty(\bbT^\infty)$, and let $\{\sigma_N\}_{N=1}^\infty$ 
be a multiplicative F{\o}lner sequence. 
Then for any function $f$, continuous on $[\ess\inf\abs{\bsymb}^2,\ess\sup\abs{\bsymb}^2]$, 
\[
\lim_{N\to\infty}
\frac1{\#{\sigma_N}}\Tr f(\bT_{\sigma_N}(\bsymb)^*\bT_{\sigma_N}(\bsymb))
=
\int_{\bbT^\infty} f(\abs{\bsymb(z)}^2)d\bm(z)
\label{a8}
\]
as $N\to\infty$. 
\end{theorem}
It is easy also to give a corollary of this Theorem in the spirit of Corollary~\ref{cr.4}; 
we will not go into details here. 

\textbf{(d)}
We give an example related to the previous theorem. Let $\gamma>1$, and let the symbol $\bsymb$ be such that 
$$
\widehat{\bsymb}(q)=
\begin{cases}
q^{-\gamma}, & q\in \bbN,
\\
0, & q\in\bbQ_+\setminus\bbN. 
\end{cases}
$$
In other words, using Bohr's lift, we have
$$
B^{-1}\bsymb(t)=\zeta(\gamma+it), \quad t\in\bbR,
$$
where $\zeta$ is the Riemann Zeta function. Then, combining the previous theorem
with Bohr's lemma \eqref{kr}, we get for every $m\in \bbN$ and for any multiplicative 
F{\o}lner sequence $\sigma_N$, 
$$
\lim_{N\to\infty}
\frac1{\#{\sigma_N}}\Tr (\bT_{\sigma_N}(\bsymb)^*\bT_{\sigma_N}(\bsymb))^m
=
\lim_{T\to\infty}\frac1{2T}\int_{-T}^T\abs{\zeta(\gamma+it)}^{2m}dt. 
$$
Observe that it is unknown whether the limit in the right hand side exists for $1/2<\gamma<1$
and $m>2$; this question is  related to deep unsolved problems in number theory, 
see e.g. \cite{Tit1986,Ram1995}.

\textbf{(e)}
Theorem~\ref{thm.a2} also implies the following corollary, an analogue of \eqref{a3}: 
\emph{If $\bsymb\in L^\infty(\bbT^\infty)$ is real and positive with 
$\inf_{\bbT^\infty}\bsymb>0$, and $\{\sigma_N\}_{N=1}^\infty$ is a multiplicative F{\o}lner sequence, 
then }
\[
\lim_{N\to\infty}(\det \bT_{\sigma_N}(\bsymb))^{1/\#\sigma_N}
=
\exp\biggl(\int_{\bbT^\infty}\log\bsymb(z)d\bm(z)\biggr).
\label{a9}
\]
Without the condition $\inf_{\bbT^\infty}\bsymb>0$, we can only assert the inequality
$$
\limsup_{N\to\infty}(\det \bT_{\sigma_N}(\bsymb))^{1/\#\sigma_N}
\leq
\exp\biggl(\int_{\bbT^\infty}\log\bsymb(z)d\bm(z)\biggr).
$$
For additive Toeplitz matrices the full version \eqref{a3} is available for 
general integrable $\symb\geq0$, which was proven initially by Szeg\H{o} \cite{Szego}
(\cite{Szego1} for positive continuous functions) and then extended by S.~Verblunsky \cite{Ver1936}
to measure generated Toeplitz matrices $T_N(\mu)$ instead of $T_N(\varphi)$; 
see  \cite{Simon} for the whole story and for at least seven other proofs of this limit formula. 
However, all of these proofs use tools from holomorphic $H^p(\bbT)$ theory, 
whose analogues for $\bbT^\infty$ are at present unknown.

\textbf{(f)}
\emph{Systems of dilated functions.} 
Let $f\in L^2(0,1)$ be given by the orthogonal series
$$
f(x)= \sum_{n=1}^\infty a_n\sqrt{2}\sin (\pi n x). 
$$
One defines
$$
\wt Bf(z)=\sum_{n=1}^\infty a_n z^{\alpha(n)}, \quad z\in \bbT^\infty,
$$
in the sense of $L^2(\bbT^\infty)$ convergence. 
This is a unitary mapping 
$$
\wt B: L^2(0,1)\to H^2(\bbT^\infty), 
$$
closely related to Bohr's lift. 
It transforms the dilation operation 
$$
D_kf(x)=f(kx), \quad k\in\bbN,
$$
into the multiplication by $z^{\alpha(k)}$, 
$$
\wt BD_k=M(z^{\alpha(k)})\wt B. 
$$
The following corollary of \textbf{(e)} above may be useful for the dilation completeness
problem (see e.g. \cite[Section 6.6.6]{Nik2018}). 

\begin{corollary}\label{cr.a5}
Let $f\in L^2(0,1)$ and let $f_k=D_k f$, $k\in\bbN$, be the dilated functions. 
Denote by $G$ the Gram matrix $\{\jap{f_j,f_k}_{L^2(0,1)}\}_{j,k\in\bbN}$ and let $G_{\sigma_N}$ be the truncation
$$
G_{\sigma_N}=\{\jap{f_j,f_k}_{L^2(0,1)}\}_{j,k\in\sigma_N}, 
$$
where $\{\sigma_N\}_{N=1}^\infty$ is a multiplicative F{\o}lner sequence. 
Assume that  $\wt Bf\in L^\infty(\bbT^\infty)$ and 
$\inf_{\bbT^\infty} \abs{\wt Bf}>0$. 
Then 
$$
\lim_{N\to\infty}(\det G_{\sigma_N})^{1/\#\sigma_N}
=
\exp\biggl(\int_{\bbT^\infty}\log \abs{\wt Bf(z)}^2d\bm(z)\biggr).
$$
\end{corollary}
\begin{proof}
Indeed, $G$ is a multiplicative Toeplitz matrix, 
\begin{align*}
\jap{f_j,f_k}_{L^2(0,1)}
&=
\jap{D_jf,D_kf}_{L^2(0,1)}
=
\jap{\wt B D_j f,\wt B D_k f}_{L^2(\bbT^\infty)}
\\
&=
\jap{M(z^{\alpha(j)})\wt B f, M(z^{\alpha(k)})\wt B f}_{L^2(\bbT^\infty)}
=
\jap{\abs{\wt B f}^2, z^{\alpha(k)-\alpha(j)}}_{L^2(\bbT^\infty)}
\\
&=
\wh\bsymb(k/j),
\end{align*}
where $\bsymb=\abs{\wt B f}^2$.
The rest follows from item \textbf{(e)} above. 
\end{proof}

\subsection{Two open problems}\label{sec.op}
\mbox{}\newline

\textbf{(a)}
Let $\sigma_N=\{1,\dots,N\}$; does the limit in \eqref{a5} exist?
Can it be expressed in a closed form in terms of $f$ and $\bsymb$? 

\textbf{(b)}
Let $\bsymb\geq0$, $\bsymb\in L^1(\bbT^\infty)$. 
Does the relation \eqref{a9} necessarily hold in this generality?
Furthermore, does it hold (similarly to the additive case) for multiplicative Toeplitz matrices of the form 
$\{\wh\bmu(j/k)\}_{j,k=1}^\infty$, where $\bmu$ is a finite measure on $\bbT^\infty$ 
with $d\bmu=\bsymb d\bm+\bmu_{\text{sing}}$?

\appendix
\section{Some proofs}\label{sec.b}

Here for completeness we give the proofs of Theorem~\ref{thm.a2} and Theorem~\ref{thm.a3}. 
We follow the set of very well known ideas, which have become 
folklore in spectral theory. In particular, our proofs deviate very little from those of \cite{Bedos}. 

\subsection{The Laurent operator $\bL(\bsymb)$}\label{sec.b1}

Along with the multiplicative Toeplitz operator $\bT(\bsymb)$, acting in $\ell^2(\bbN)$, 
we consider the ``multiplicative Laurent operator" 
$$
\bL(\bsymb)=\{\wh\bsymb(q/s)\}_{q,s\in\bbQ_+}\quad \text{ in $\ell^2(\bbQ_+)$.}
$$
(In the additive case, the Laurent operator is $\{\wh\symb(j-k)\}_{j,k\in\bbZ}$ in
$\ell^2(\bbZ)$, see \cite[Section 1.2]{BS}.)
Clearly, $\bL(\bsymb)$ is the matrix of the operator of multiplication by $\bsymb$ 
in the basis $\{\be_\alpha\}_{\alpha\in\bbZ^{(\infty)}}$ in $L^2(\bbT^\infty)$:
$$
\jap{\bsymb\cdot \be_\alpha,\be_\beta}_{L^2(\bbT^{\infty})}
=
\wh \bsymb(p^{\beta-\alpha}), 
\quad
\alpha,\beta\in \bbZ^{(\infty)}.
$$
From here it follows that $\bL(\bsymb)$ is bounded for $\bsymb\in L^\infty(\bbT^\infty)$ and has the following simple properties, which we state here for the ease of further reference:

\begin{lemma}\label{lma.b1}
\begin{enumerate}[\rm (i)]
\item
If $\bsymb_1,\bsymb_2\in L^\infty(\bbT^\infty)$, then 
$$
\bL(\bsymb_1)\bL(\bsymb_1)=\bL(\bsymb_1 \bsymb_2). 
$$
\item
If $f$ is a polynomial and $\bsymb\in L^\infty(\bbT^\infty)$, then 
$$
f(\bL(\bsymb))=\bL(f\circ \bsymb).
$$
\item
If $\bsymb\in L^\infty(\bbT^\infty)$, then 
$$
\bL(\bsymb)^*=\bL(\overline{\bsymb}).
$$
\end{enumerate}
\end{lemma}

\subsection{Main lemma}\label{sec.b2}
First we need to extend the F{\o}lner condition \eqref{a5a} to all $n\in\bbQ_+$. 
\begin{lemma}\label{lma.b2a}
Let $\{\sigma_N\}_{N=1}^\infty$ be a multiplicative F{\o}lner sequence of finite
subsets of $\bbN$. Then \eqref{a5a} holds for any $n\in\bbQ_+$. 
\end{lemma}
\begin{proof}
Let $n=a/b$, where $a,b\in\bbN$ are coprime. 
If $nk\in\sigma_N$, then $b$ divides $k$, and so, denoting $k=bj$ and using
counting arguments, we get
\begin{align*}
\#\{k\in\sigma_N: ak/b\in\sigma_N\}
&=
\#\{j\in\bbN: ja\in \sigma_N, jb\in \sigma_N\}
\\
&\geq
\#\{j\in\sigma_N: ja\in \sigma_N, jb\in \sigma_N\}
\\
&=
\#\{j\in\sigma_N: ja\in \sigma_N\}
-
\#\{j\in\sigma_N: ja\in \sigma_N, jb\notin \sigma_N\}
\\
&\geq
\#\{j\in\sigma_N: ja\in \sigma_N\}
-
\#\{j\in\sigma_N: jb\notin \sigma_N\}
\\
&=
\#\{j\in\sigma_N: ja\in \sigma_N\}
-
\#\sigma_N
+
\#\{j\in\sigma_N: jb\in \sigma_N\}.
\end{align*}
Dividing by $\#\sigma_N$ and using  condition \eqref{a5a}  for $n=a$ and $n=b$, 
we obtain
$$
\liminf_{N\to\infty}\frac{\#\{k\in\sigma_N: ak/b\in\sigma_N\}}{\#\sigma_N}\geq1,
$$
which implies \eqref{a5a} for $n=a/b$. 
\end{proof}

For $\sigma_N\subset\bbN$, we consider $\ell^2(\sigma_N)$ as a subspace of $\ell^2(\bbQ_+)$; 
let $\pi_N$ be the orthogonal projection in $\ell^2(\bbQ_+)$ with the range $\ell^2(\sigma_N)$. 
\begin{lemma}\label{lma.b2}
Let $\bsymb\in L^\infty(\bbT^\infty)$, let $\{\sigma_N\}_{N=1}^\infty$ 
be a multiplicative F{\o}lner sequence, and 
let $\pi_N$ be as above. 
Then 
\[
\frac1{\#{\sigma_N}}\norm{\pi_N \bL(\bsymb)(I-\pi_N)}_{\Sch_2}^2\to0,\quad N\to\infty.
\label{b5}
\]
\end{lemma}
\begin{proof}
We have 
\begin{align}
\frac1{\#{\sigma_N}}\norm{\pi_N \bL(\bsymb)(I-\pi_N)}_{\Sch_2}^2
&=
\frac1{\#{\sigma_N}}\sum_{n\in\sigma_N}\sum_{q\in\bbQ_+\setminus \sigma_N}\abs{\wh\bsymb(n/q)}^2
\notag
\\
&=
\frac1{\#{\sigma_N}}\sum_{r\in\bbQ_+}\abs{\wh\bsymb(r)}^2
\#{\{n\in\sigma_N: n/r\in\bbQ_+\setminus\sigma_N\}}
\notag
\\
&=
\sum_{r\in\bbQ_+}\abs{\wh\bsymb(r)}^2
\biggl(1-\frac{\#{\{n\in\sigma_N: n/r\in\sigma_N\}}}{\#{\sigma_N}} \biggr).
\label{b5a}
\end{align}
We have $\bsymb\in L^\infty(\bbT^\infty)\subset L^2(\bbT^\infty)$, and therefore
$\sum_{r\in\bbQ_+}\abs{\wh \bsymb(r)}^2<\infty$. 
By Lemma~\ref{lma.b2a}, the term in brackets in the 
r.h.s. of \eqref{b5a} converges to zero as $N\to\infty$ for all $r\in\bbQ_+$. 
Now \eqref{b5} follows by dominated convergence. 
\end{proof}

Below we will use Lemma~\ref{lma.b2} in combination with 
the following simple estimate
(which can be seen as an elementary version of the general theorem of \cite{LaSa}).

\begin{proposition}\label{lma.b3}
Let $L$ be a bounded self-adjoint operator in a Hilbert space, and let $\pi$ be an orthogonal
projection in the same space such that $\pi L(I-\pi)\in\Sch_2$. 
Then for any $n\geq2$ one has the trace norm estimate
\[
\norm{\pi L^n\pi -(\pi L\pi)^n}_{\Sch_1}
\leq 
\frac{n(n-1)}{2}\norm{L}^{n-2}\norm{\pi L(1-\pi)}_{\Sch_2}^2. 
\label{b6}
\]
\end{proposition}
\begin{proof}
Denote $\pi^\perp=I-\pi$. 
First by induction in $n\geq1$, one easily proves the estimate
\[
\norm{\pi L^n\pi^\perp}_{\Sch_2}\leq n\norm{L}^{n-1}\norm{\pi L\pi^\perp}_{\Sch_2}.
\label{b1}
\]
Using this, it is easy to prove \eqref{b6} by induction in $n\geq2$. 
The key step is to write
$$
\pi L^{n+1}\pi -(\pi L\pi )^{n+1}
=
(\pi L^n \pi -(\pi L\pi )^n)(\pi L\pi )
+
(\pi L^n\pi^\perp)(\pi^\perp L\pi)
$$
and to estimate the first term by using the induction hypothesis, 
$$
\norm{(\pi L^n \pi -(\pi L\pi )^n)(\pi L\pi )}_{\Sch_1}
\leq 
\norm{L}\norm{\pi L^n \pi -(\pi L\pi )^n}_{\Sch_1}
\leq 
\tfrac{n(n-1)}{2}\norm{L}^{n-1}\norm{\pi L\pi^\perp}_{\Sch_2}^2
$$
and the second term using \eqref{b1}, 
$$
\norm{(\pi L^n\pi^\perp)(\pi^\perp L\pi)}_{\Sch_1}
\leq 
\norm{\pi L^n\pi^\perp}_{\Sch_2} \norm{\pi L\pi^\perp}_{\Sch_2} 
\leq
n\norm{L}^{n-1}\norm{\pi L\pi^\perp}_{\Sch_2}^2.
$$
Now it remains to combine this and notice that $\tfrac{n(n-1)}{2}+n=\tfrac{(n+1)n}{2}$.
\end{proof}

\subsection{Proof of Theorem~\ref{thm.a2}}\label{sec.b3}

The proof proceeds in two steps. 

\emph{1) Let $f$ be a polynomial.}
Clearly, the theorem holds if $f$ is a constant. Thus, subtracting a constant, 
we can always reduce the problem to the case $f(0)=0$. 
Now our aim is to apply Proposition~\ref{lma.b3} with $L=\bL(\bsymb)$ and 
$\pi=\pi_N$. 

Consider the term $\pi_N f(\bL(\bsymb)) \pi_N$. 
By Lemma~\ref{lma.b1}(ii), we have
\[
f(\bL(\bsymb))=\bL(f\circ \bsymb).
\label{b7}
\]
Further, it is clear that for any $\bsymb\in L^\infty(\bbT^\infty)$, we have
$$
\Tr (\pi_N \bL(\bsymb)\pi_N)
=
(\#{\sigma_N})\wh\bsymb(1)
=
(\#{\sigma_N})\int_{\bbT^\infty}\bsymb(z)d\bm(z).
$$
Let us apply this with $f\circ \bsymb$ in place of $\bsymb$: 
$$
\Tr (\pi_N \bL(f\circ \bsymb)\pi_N)
=
(\#{\sigma_N})
\int_{\bbT^\infty}f(\bsymb(z))d\bm(z).
$$
Combining with \eqref{b7}, we obtain
\[
\Tr (\pi_N f(\bL(\bsymb))\pi_N)
=
(\#{\sigma_N})
\int_{\bbT^\infty}f(\bsymb(z))d\bm(z).
\label{b8}
\]

Consider the term $f(\pi_N \bL(\bsymb)\pi_N)$. 
It is clear that
$$
\pi_N \bL(\bsymb) \pi_N=\bT_\sigma(\bsymb)\oplus 0
$$
with respect to the orthogonal decomposition $\ell^2(\bbQ_+)=\ell^2(\sigma_N)\oplus \ell^2(\bbQ_+\setminus \sigma_N)$. 
From here, using the condition $f(0)=0$, we obtain 
$$
\Tr f(\pi_N \bL(\bsymb)\pi_N)=\Tr f(\bT_{\sigma_N}(\symb)).
$$
Combining this with \eqref{b8} and applying Proposition~\ref{lma.b3}, we get
$$
\frac1{\#{\sigma_N}}
\Abs{\Tr f(\bT_{\sigma_N}(\bsymb))
-
\int_{\bbT^\infty}f(\bsymb(z))d\bm(z)}
\leq
C(f)
\frac1{\#{\sigma_N}}
\norm{\pi_N \bL(\bsymb)(I-\pi_N)}_{\Sch_2}^2. 
$$
Finally, by  Lemma~\ref{lma.b2} the right hand side here tends to zero as $N\to\infty$.  
This proves Theorem~\ref{thm.a2} for polynomial $f$. 

\emph{2) The case of a general $f$.} 
By considering the real and imaginary parts of $f$ separately, we reduce the problem 
to the case when $f$ is real valued. Next, 
we use a standard approximation argument. 
By a variational argument, the spectra of $\bT(\bsymb)$ and $\bT_{\sigma_N}(\bsymb)$ are
contained in the interval $[\ess\inf \bsymb,\ess\sup\bsymb]$. 
Now let $f_+$ and $f_-$ be polynomials with real coefficients such that 
$$
f_-(t)\leq f(t)\leq f_+(t),\quad t\in[\ess\inf \bsymb,\ess\sup\bsymb], 
$$
and 
\[
0\leq f_+(t)-f_-(t)\leq \eps, \quad t\in[\ess\inf \bsymb,\ess\sup\bsymb].
\label{b8a}
\]
Then 
$$
\Tr f_-(\bT_{\sigma_N}(\bsymb))
\leq 
\Tr f(\bT_{\sigma_N}(\bsymb))
\leq
\Tr f_+(\bT_{\sigma_N}(\bsymb))
$$
and so, by the previous step of the proof,
$$
\limsup_{N\to\infty}\frac1{\#{\sigma_N}}\Tr f(\bT_{\sigma_N}(\bsymb))
\leq
\int_{\bbT^\infty} f_+(\bsymb(z))d\bm(z),
$$
and similarly 
$$
\liminf_{N\to\infty}\frac1{\#{\sigma_N}}\Tr f(\bT_{\sigma_N}(\bsymb))
\geq
\int_{\bbT^\infty} f_-(\bsymb(z))d\bm(z)\, .
$$
On the other hand, by \eqref{b8a}, 
$$
\int_{\bbT^\infty}\bigl(f_+(\bsymb(z))-f_-(\bsymb(z))\bigr)d\bm(z)\leq \eps. 
$$
Since $\eps>0$ is arbitrary, it follows that 
\begin{multline*}
\limsup_{N\to\infty}\frac1{\#{\sigma_N}}\Tr f(\bT_{\sigma_N}(\bsymb))
=
\liminf_{N\to\infty}\frac1{\#{\sigma_N}}\Tr f(\bT_{\sigma_N}(\bsymb))
=
\int_{\bbT^\infty} f(\bsymb(z))d\bm(z),
\end{multline*}
as required. 
The proof of Theorem~\ref{thm.a2} is complete. 

\subsection{Proof of Theorem~\ref{thm.a3}}

The idea is to replace $\bT_{\sigma_N}(\bsymb)^*\bT_{\sigma_N}(\bsymb)$ 
by $\bT_{\sigma_N}(\abs{\bsymb}^2)$ in \eqref{a8} 
and thereby to reduce the problem to Theorem~\ref{thm.a2}. 

\emph{1) Let $f$ be a polynomial.}
First we need an auxiliary estimate. 
Let $X$ and $Y$ be bounded operators in a Hilbert space 
with operator norms satisfying $\norm{X}\leq B$, $\norm{Y}\leq B$
for some $B>0$. 
For any $m\in\bbN$, as in \eqref{a7}, we have
$$
\norm{X^m-Y^m}_{\Sch_1}
\leq
mB^{m-1}
\norm{X-Y}_{\Sch_1}. 
$$
It follows that for any polynomial $f$, 
\[
\norm{f(X)-f(Y)}_{\Sch_1}\leq C(f,B)\norm{X-Y}_{\Sch_1}. 
\label{b9}
\]
(Of course, much more refined estimates of this kind are available, 
but we prefer to use the most elementary tools here.)

Next, we have
$$
\bT_{\sigma_N}(\abs{\bsymb}^2)-\bT_{\sigma_N}(\bsymb)^*\bT_{\sigma_N}(\bsymb)
=
\bigl(\pi_N\bL(\overline{\bsymb}\bsymb)\pi_N
-
\pi_N\bL(\bsymb)^*\pi_N\bL(\bsymb)\pi_N\bigr)|_{\ell^2(\sigma_N)},
$$
and 
using Lemma~\ref{lma.b1},
\begin{multline*}
\pi_N\bL(\overline{\bsymb}\bsymb)\pi_N
-
\pi_N\bL(\bsymb)^*\pi_N\bL(\bsymb)\pi_N
\\
=
\pi_N\bL(\bsymb)^*\bL(\bsymb)\pi_N
-
\pi_N\bL(\bsymb)^*\pi_N \bL(\bsymb)\pi_N
=
\pi_N\bL(\bsymb)^*(I-\pi_N) \bL(\bsymb)\pi_N.
\end{multline*}
It follows that 
$$
\norm{\bT_{\sigma_N}(\abs{\bsymb}^2)-\bT_{\sigma_N}(\bsymb)^*\bT_{\sigma_N}(\bsymb)}_{\Sch_1}
=
\norm{\pi_N \bL(\bsymb)^*(I-\pi_N)}_{\Sch_2}^2. 
$$
Combining this with \eqref{b9}, we get
\begin{multline*}
\frac1{\#{\sigma_N}}
\Abs{\Tr f(\bT_{\sigma_N}(\abs{\bsymb}^2))-\Tr f(\bT_{\sigma_N}(\bsymb)^*\bT_{\sigma_N}(\bsymb))}
\\
\leq
C(f,\norm{\bsymb}_{L^\infty})
\frac1{\#{\sigma_N}}
\norm{\pi_N \bL(\bsymb)^*(I-\pi_N)}_{\Sch_2}^2. 
\end{multline*}
By Lemma~\ref{lma.b2}, the right hand side here goes to zero as $N\to\infty$. 
Thus, the claim follows from 
Theorem~\ref{thm.a2}, with $\abs{\bsymb}^2$ in place of $\bsymb$. 

\emph{2) The case of a general $f$.} Here the proof proceeds by approximation argument, 
in exactly the same way as in Theorem~\ref{thm.a2}.


\end{document}